\documentclass[11pt]{article}

\usepackage[margin=1in]{geometry}
\usepackage{amsmath,amssymb,amsthm,mathtools}
\usepackage{enumitem, comment}
\usepackage[hidelinks]{hyperref}

\newtheorem{theorem}{Theorem}
\newtheorem{lemma}{Lemma}
\newtheorem{proposition}{Proposition}

\newtheorem{remark}{Remark}

\newcommand{\B}{\mathbb B}
\newcommand{\C}{\mathbb C}
\newcommand{\PP}{\mathbb P}
\newcommand{\ii}{\sqrt{-1}}
\newcommand{\ord}{\operatorname{ord}}
\newcommand{\trdeg}{\operatorname{trdeg}}

\title{Rigidity for local holomorphic isometric maps from the complex unit ball to reducible
bounded symmetric domains}

\author{Yuan Yuan; Xu Zhang}
\date{}

\begin{document}

\maketitle

\begin{abstract}
Let \(U\) be a connected open subset of the complex unit ball
\(\B^n\), \(n\geq2\), and let \(F_i:U\to\Omega_i\) be non-constant
holomorphic maps into irreducible bounded symmetric domains.  We prove
that if $F=(F_1, \cdots, F_m)$ is a holomorphic isometric map from $U$ to the product of $\Omega_i$, then
 every  \(F_i\) extends to a proper holomorphic isometric embedding up to a
positive constant from \(\B^n\) into \(\Omega_i\).
\end{abstract}

\section{Introduction}
A local holomorphic map that preserves a canonical metric is subject to
two rather different constraints: it may be forced to continue beyond
its initial domain, and its image may be forced to have a prescribed
geometric structure.  Calabi's work \cite{C53} provided an early
framework for both phenomena in the real-analytic K\"ahler category.
For bounded symmetric domains, these questions acquire an additional
splitting aspect.  If \(D\) is irreducible, the target is a product
\(\Omega_1\times\cdots\times\Omega_m\), and a holomorphic map 
\(F=(F_1,\ldots,F_m)\) defined on a connected open subset $U \subset D$ obeys
\[
 \omega_D=\sum_{i=1}^m\lambda_iF_i^*\omega_{\Omega_i},
\]
on $U$ for constant $\lambda_i>0$, 
the metric identity concerns the map \(F\) as a whole.  
Mok's deep and influential  extension theorem then promotes a
local map to a global proper isometric embedding from $D$ to \(\Omega_1\times\cdots\times\Omega_m\). 
\cite{M12}. 
The central
issue is whether it also determines the metric contribution of each
individual factor.

The rank of the source separates two qualitatively different regimes.
Suppose first that \(D\) is irreducible and has rank at least two.
Clozel--Ullmo observed that Mok's Hermitian metric rigidity can be
applied to show total geodesy for a holomorphic isometry from \(D\) to
an arbitrary bounded symmetric domain, including a reducible one
\cite{CU03}.   Thus, for higher-rank sources, the strongest natural
rigidity conclusion is already available.

No parallel statement can hold for a rank-one source, namely  a complex ball, and it supports isometric embeddings that are not
totally geodesic.  Mok produced examples from the disk to reducible
bounded symmetric domains and later constructed, for every irreducible
higher-rank bounded symmetric domain \(\Omega\), an unexpected proper nonstandard
isometric embedding from \(\mathbb B^d \) to \(\Omega\) for an appropriate
dimension \(d\) using the theory of VMRT  \cite{M12, M16}.  
Any
general theorem for a ball must leave room for these nonstandard maps.
What can still be rigid is the decomposition of the metric as one may ask
that every nonconstant component be an isometry after rescaling,
without demanding that it be totally geodesic.  This is the optimal 
conclusion compatible with Mok's constructions.

Earlier work established this componentwise principle for important
classes of targets.  Yuan--Zhang treated products of balls when the
source is \(\mathbb B^n\), \(n\geq2\), and showed that each
nonconstant factor is totally geodesic \cite{YZ12}.  Their
argument passes through algebraic continuation and then uses boundary
behavior together with a CR linearity theorem.  Xiao later 
obtained the desired factorwise conclusion for products whose irreducible
members are balls or Lie balls \cite{X22}.  The interested reader may also refer to related results such as \cite{M02, N11, MN12, HY14, CM17, M18, UWZ19, XY20, FHX20, D23, CHYZ25, MW26, D26}.

\medskip

We solve the general problem (cf. Problem 5.1 in \cite{Y19}) in this paper.  
\begin{theorem}\label{thm:main}
Let \(n\geq2\), let \(U\subset\B^n\) be connected and open, and let
\(\Omega_i\) be irreducible bounded symmetric domains in their
Harish--Chandra realizations for $1\leq i\leq m$.  Suppose that
$ F_i:U\rightarrow\Omega_i $
is a non-constant holomorphic map for each $i$, and
\begin{equation}\label{eq:main-assumption}
 \omega_{\B^n}
 =\sum_{i=1}^m\lambda_iF_i^*\omega_{\Omega_i},
\end{equation}
holds for constant $\lambda_i>0$. 
Let \(\gamma_i\) denote the genus of \(\Omega_i\).  Then there are
integers \(k_i\geq1\) such that each \(F_i\) extends uniquely to a proper holomorphic
isometric embedding satisfying 
\begin{equation}\label{eq:main-quantization}
 F_i^*\omega_{\Omega_i}
 =\frac{\gamma_i k_i}{n+1}\,\omega_{\B^n},
\end{equation}
with $ \sum_{i=1}^m \lambda_i\frac{\gamma_i k_i}{n+1}=1.$
\end{theorem}

Note that there is  no
restriction on the type, rank, or dimension of the irreducible bounded
symmetric domain \(\Omega_i\) in Theorem \ref{thm:main}.  
Even when all $\lambda_i$ are equal, (\ref{eq:main-quantization}) becomes the equation of  holomorphic isometry with respect to some conformal constant from $\mathbb{B}^n$ to \(\Omega_1\times\cdots\times\Omega_m\) with respect to Bergman metrics and the theorem is unknown before. 
The
condition \(n\geq2\) is necessary as shown in Mok's example of $p$-th root embedding \cite{M12}. The phenomenon also appears here since the divisor used in
the proof can split for a disk and thereby permits root-type
isometries. Nevertheless, the problem for the disk is extensively studied by many authors (cf. \cite{MN09, N10, C16, C17, CY19}).

Our proof places the results of Yuan--Zhang and Xiao in one common
scheme (for any and all types of the target bounded symmetric domains), and it does not rely on CR boundary geometry.  For each
irreducible factor \(\Omega_i\), its generic norm is a finite Hermitian
polynomial and a signature decomposition of that polynomial realizes
\(\Omega_i\) inside a generalized ball endowed with an indefinite
canonical form, using the idea in \cite{XY20}.  After composing with these realizations, it remains
to split an isometry from \(\mathbb B^n\) to a product of generalized
balls.  Finite Hermitian rank first leads to algebraicity of the component 
functions.  We then place their branches on finite normal covers of
projective space.  Bertini irreducibility 
identifies the lifted incidence hypersurface as a prime divisor occurring with multiplicity one, while logarithmic residues eliminate every other mixed
divisor.  Separating the remaining vertical and horizontal contributions
then shows that every target potential function equals the source potential function raised to a
nonnegative integer exponent.  In this
way the key input is algebraic geometry rather than CR geometry,
and all irreducible target domains are handled by the same argument.

\medskip

We finish the introduction by recalling some terminology. 
For \(0\leq\kappa<M\), let
\[
 H_\kappa(\xi,v)
 =-\sum_{\alpha=1}^{\kappa}\xi_\alpha v_\alpha
  +\sum_{\alpha=\kappa+1}^{M}\xi_\alpha v_\alpha
\]
be the complex-bilinear polarization of the relevant Hermitian form, 
with the value   \(H_\kappa(\xi,\overline\xi)\) on the diagonal.
Define the generalized ball
\[
 \B^M_\kappa
 =\{\xi\in\C^M:1-H_\kappa(\xi,\overline\xi)>0\}
\]
with the indefinite canonical form 
$ \eta_{\kappa,M} =-\ii\,\partial\bar\partial   \log\bigl(1-H_\kappa(\xi,\overline\xi)\bigr).$
When \(\kappa=0\), this is the usual complex hyperbolic form. 
 In particular,
$ \omega_{\B^n}=(n+1)\eta_{0,n}.$
Here, $ \omega_D=\ii\,\partial\bar\partial\log K_D(z,\overline z)$ denotes the Bergman form and  \(K_D\)  the Bergman kernel 
for a bounded domain \(D\).

\section{A rigidity result for generalized balls}

The next theorem is the key to proving Theorem \ref{thm:main}, which generalizes the rigidity result in \cite{YZ12} to generalized balls. 

\begin{theorem}\label{thm:generalized-balls}
Let \(n\geq2\), let \(U\subset\B^n\) be a connected  open set, and let
$ G_i:U\rightarrow\B^{M_i}_{\kappa_i}$
be a holomorphic map with $0\leq\kappa_i<M_i$.  If
\begin{equation}\label{eq:generalized-assumption}
 \eta_{0,n}
 =\sum_{i=1}^m a_iG_i^*\eta_{\kappa_i,M_i}
\end{equation}
holds on $U$ for $a_i>0$, 
then there are integers \(k_i\geq0\) such that
\begin{equation}\label{eq:generalized-conclusion}
 G_i^*\eta_{\kappa_i,M_i}=k_i\eta_{0,n}
\end{equation}
and $ \sum_{i=1}^ma_i k_i=1.$
\end{theorem}

\subsection{Normalization and finite Hermitian rank}

Fix a point  \(p \in U\).  The homogeneous Hermitian form associated with
\(\B^M_\kappa\) is
\[
 \mathcal H_\kappa((s,\xi),(t,\upsilon))
 =s\overline t-H_\kappa(\xi,\overline\upsilon),
\]
of signature \((\kappa+1,M-\kappa)\).  Since the pseudo-unitary group is
transitive on the positive projective lines, a projective
pseudo-unitary transformation sends the line \([1:G_i(p)]\) to
\([1:0]\).  In the affine chart this is a fractional linear map defined
on a neighborhood of \(G_i(p)\), and it satisfies
\[
 1-H_{\kappa_i}(T_i(\xi),\overline{T_i(\xi)})
 =\frac{1-H_{\kappa_i}(\xi,\overline\xi)}{|j_i(\xi)|^2}
\]
for a nonvanishing affine-linear denominator \(j_i\).  
Together with an
automorphism of the source ball sending \(p\) to the origin, we may assume
$ 0\in U, G_i(0)=0.$
By the standard polarization argument, we
introduce independent complex variables \(z,w\in\C^n\), and write
$ K(z,w)=1-z\cdot w$ with
$ z\cdot w=\sum_{\nu=1}^n z_\nu w_\nu$.
For a scalar holomorphic germ \(h\), set
$ h^\#(w)=\overline{h(\overline w)}.$
For vector-valued germs, \(^{\#}\) is applied componentwise.
The polarized target potential functions are
$ Q_i(z,w)=1-H_{\kappa_i}\bigl(G_i(z),G_i^\#(w)\bigr)$
and satisfy
$ K(z,0)=K(0,w)=Q_i(z,0)=Q_i(0,w)=1.$
By the standard reduction as in \cite{CU03} and the polarization, we get
\begin{equation}\label{eq:polarized-log}
 \log K(z,w)=\sum_{i=1}^ma_i\log Q_i(z,w)
\end{equation}
near \((0,0)\), with all logarithm branches normalized to vanish
there.  Equivalently,
\begin{equation}\label{eq:polarized-dlog}
 d\log K=\sum_{i=1}^ma_i\,d\log Q_i.
\end{equation}
Since every \(Q_i\) has finite Hermitian rank, by Theorem
3.2 in \cite{U88}, we may write 
\begin{equation}\label{eq:minimal-kernel}
 Q_i(z,w)
 =1+\sum_{\alpha=1}^{r_i}
   \varepsilon_{i\alpha}
   f_{i\alpha}(z)f_{i\alpha}^\#(w),
\end{equation}
for $ \varepsilon_{i\alpha}\in\{1,-1\}$,
where $ f_{i\alpha}(0)=0$ and
$ \{1,f_{i1},\ldots,f_{ir_i}\}$
 is linearly independent. 

\subsection{Algebraicity}

The algebraicity argument originates in \cite{HY14, CHYZ25} and we follow the original idea developed in \cite{HY14} here. Note that the linear independence is crucial.

\begin{proposition}\label{lem:effective-algebraicity}
Every function \(f_{i\alpha}\) in
\eqref{eq:minimal-kernel} is algebraic over the rational function
field \(\C(z)\).
\end{proposition}

\begin{proof}
Let \(E\) be the field generated over \(\C(z)\) by all the functions
\(f_{i\alpha}\), and suppose
$ r=\trdeg_{\C(z)}E>0.$
Choose a transcendence basis
\(\phi_1,\ldots,\phi_r\) from among these functions and set
\(\phi=(\phi_1,\ldots,\phi_r)\).  Introduce independent variables
\(X=(X_1,\ldots,X_r)\) and set \(\widehat\phi_j(z,X)=X_j\).
Every remaining \(f_{i\alpha}\) is algebraic over
\(\C(z)(\phi_1,\ldots,\phi_r)\).  Choose its minimal polynomial and
replace \(\phi_j\) in its coefficients by \(X_j\).  Since the
extension is separable in characteristic zero, away from the
discriminant the implicit function theorem selects a local algebraic
branch $ \widehat f_{i\alpha}(z,X)$
near a generic point of the graph \(X=\phi(z)\), such that
\[
 f_{i\alpha}(z)
 =\widehat f_{i\alpha}(z,\phi(z)).
\]
After shrinking this neighborhood, put all these branches in a single
finite extension \(L\) of \(\C(z,X)\).  The polarized identity, first
obtained near \((0,0)\), continues to this generic graph neighborhood
by the identity theorem.

Define
\[
 \widehat Q_i(z,X,w)
 =1+\sum_\alpha\varepsilon_{i\alpha}
   \widehat f_{i\alpha}(z,X)f_{i\alpha}^\#(w)
\]
and
\[
 \Psi(z,X,w)
 =\log K(z,w)-\sum_i a_i\log\widehat Q_i(z,X,w).
\]
For every \(w\)-multiindex \(\beta\), set
\[
 A_\beta(z,X)
 =\left.\partial_w^\beta\Psi(z,X,w)\right|_{w=0}
\]
(so \(A_\beta\) is a Taylor derivative, without factorial
normalization).  For \(\beta=0\), one has \(A_0=0\) identically because
\(K(z,0)=\widehat Q_i(z,X,0)=1\).  If \(|\beta|>0\), differentiation of
the logarithms at \(w=0\) expresses \(A_\beta\) rationally in the
algebraic functions \(\widehat f_{i\alpha}\); hence \(A_\beta\in L\).
It vanishes on the selected branch of the graph \(X=\phi(z)\) by
\eqref{eq:polarized-log}.

Suppose that some \(A_\beta\) is nonzero in \(L\), and let
$ T^d+c_{d-1}(z,X)T^{d-1}+\cdots+c_0(z,X)$
be its monic minimal polynomial over \(\C(z,X)\).  Substitution of the
selected branch \(X=\phi(z)\), away from the poles of the coefficients,
yields \(c_0(z,\phi(z))=0\).  The substitution homomorphism
\(\C(z,X)\to E\), given by \(X_j\mapsto\phi_j(z)\), is injective by algebraic
independence, so \(c_0=0\).  Irreducibility then forces the minimal
polynomial to be \(T\), contradicting that \(A_\beta\neq0\).  Thus every
\(A_\beta\) vanishes.  Taylor expansion in \(w\) now yields an identity
with \(X\) independent.  Taking the \emph{total} exterior differential
in \((z,X,w)\) yields
\begin{equation}\label{eq:lifted-identity}
 d\log K=\sum_i a_i\,d\log\widehat Q_i
\end{equation}
on that neighborhood.

Because \(r>0\) and \(\widehat\phi_j=X_j\), at least one lifted function
genuinely depends on \(X\).  Choose a generic \(z=z^0\) and an
irreducible algebraic curve in the \(X\)-space on which one such
function remains nonconstant.  Restrict all lifted functions to this
curve and normalize its projective closure in the common finite
function field induced by \(L\).  On the resulting compact normal
curve, some nonconstant \(\widehat f_{i\alpha}(z^0,X)\) has a pole at a
point \(p\).  The restriction of \eqref{eq:lifted-identity}, valid near
the original point of the curve, extends by the identity theorem as an
identity of meromorphic one-forms on this compact normalization.

Let \(t\) be a local parameter at \(p\).  For an index \(i\) having a
polar component, let \(s_i>0\) be the largest pole order among the
\(\widehat f_{i\alpha}\), and write \(c_{i\alpha}\) for the coefficient
of \(t^{-s_i}\).  The leading coefficient of \(\widehat Q_i\) is
\[
 \ell_i(w)=\sum_{\{\alpha:\,\ord_p\widehat f_{i\alpha}=-s_i\}}
 \varepsilon_{i\alpha}c_{i\alpha}f_{i\alpha}^{\#}(w).
\]
It is not identically zero, because not all \(c_{i\alpha}\) vanish and
the germs \(f_{i\alpha}^{\#}\) are linearly independent. 
If no component for a given \(i\) has
a pole, the limiting value of \(\widehat Q_i\) at \(p\) is an analytic
function of \(w\) equal to \(1\) at \(w=0\), and hence is not
identically zero.  There are only finitely many indices.  We may
therefore choose a single generic \(w\), arbitrarily close to zero,
for which every relevant leading or limiting coefficient is nonzero.
For this \(w\),
$ \ord_p\widehat Q_i\leq0$ for all $i,$
 and $ \ord_p\widehat Q_j<0$ for at least one $j.$
Restrict the total-differential identity
\eqref{eq:lifted-identity} to the curve with \(z=z^0\) and this \(w\)
fixed.  Since \(K(z^0,w)\) is independent of the curve variable, the
residue at \(p\) yields
$ 0=\sum_i a_i\ord_p\widehat Q_i<0,$
because every \(a_i\) is positive.  This contradiction yields
\(r=0\).
\end{proof}

\subsection{Proof of Theorem \ref{thm:generalized-balls}}
The following lemma is a consequence of Bertini's irreducibility
theorem and we include the argument needed for the lifted incidence
divisor.

\begin{lemma}\label{lem:prime-incidence}
Let \(n\geq2\) and \(X, Y\) be irreducible normal projective varieties with
finite surjective morphisms
$ \pi_X:X\rightarrow\PP^n,$
$ \pi_Y:Y\rightarrow\PP^n.$
Define \begin{equation*}
 \mathcal I
 =\left\{
   Z_0W_0-\sum_{\nu=1}^nZ_\nu W_\nu=0
  \right\}
 \subset\PP^n\times\PP^n
\end{equation*}
to be the incidence variety. 
Then the pullback  of $ \mathcal I$
 to \(X\times Y\) given by  
 $$\widetilde{\mathcal I} := (\pi_X\times\pi_Y)^{-1}(\mathcal I) \subset X\times Y$$
is an irreducible, generically reduced divisor.
\end{lemma}

\begin{proof}
Let
\[
 B(Z,W)=Z_0W_0-\sum_{\nu=1}^n Z_\nu W_\nu,
\]
viewed as a section of
\(\mathcal O_{\PP^n\times\PP^n}(1,1)\), and let
\[
 \widetilde B
 :=
 (\pi_X\times\pi_Y)^*B.
\]
Then
$ \widetilde{\mathcal I}=V(\widetilde B).$
Projecting \(\widetilde{\mathcal I}\) to \(X\), the fiber over
\(x\in X\) is
\[
 \left\{y\in Y:
 B\bigl(\pi_X(x),\pi_Y(y)\bigr)=0\right\}
 =
 \pi_Y^{-1}(H_{\pi_X(x)}),
\]
where
\[
 H_{\pi_X(x)}
 =
 \left\{[W]\in\PP^n:
 \pi_X(x)_0W_0-
 \sum_{\nu=1}^n\pi_X(x)_\nu W_\nu=0\right\}.
\]
Since \(\pi_X\) is surjective, these are all hyperplanes.  Apply
Theorem 3.3.1 in \cite{L04} to
\(\pi_Y:Y\to\PP^n\) with \(d=1\).  Since
\(\dim\pi_Y(Y)=n\geq2\), the inverse image of a general hyperplane is
irreducible.  The Bertini-open set of such hyperplanes pulls back under
the finite surjective map \(\pi_X\) to a dense open subset of \(X\).
We will show that 
it is also generically reduced.  The finite extension
\(\C(Y)/\C(\PP^n)\) is separable in characteristic zero, so
\(\pi_Y\) is \'etale over the complement of a proper branch locus
\( \mathcal B \subset\PP^n\).  A general hyperplane is not contained in \(\mathcal B\);
therefore its generic point lies outside \(\mathcal B\), and its inverse image
is reduced at its generic point.

 Since 
\(\pi_X\times\pi_Y\) is surjective, \(B\) is not identically zero and thus the section \(\widetilde B\) is not identically zero. 
Since \(X\times Y\) is integral, \(V(\widetilde B)\) is an effective
Cartier divisor.
If a component were of the form \(D\times Y\), where
\(D\subset X\) is an irreducible divisor, 
 then for general \(x\in D\)
one would have
$ B\bigl(\pi_X(x),\pi_Y(y)\bigr)=0$
for all $y\in Y.$ 
The surjectivity of \(\pi_Y\) would imply that the nonzero linear form
\(B(\pi_X(x),W)\) vanishes on all of \(\PP^n\), which is impossible.
Let \(C\) be an irreducible component of
\(\widetilde{\mathcal I}\). Since \(\widetilde{\mathcal I}\) is an
effective Cartier divisor in the integral variety \(X\times Y\),
every irreducible component of \(\widetilde{\mathcal I}\) has dimension
\(2n-1\).
Suppose that \(C\) does not dominate \(X\), and set
$ D:=p_X(C).$
Since \(p_X\) is proper, \(D\) is a closed irreducible proper
subvariety of \(X\), so \(\dim D\leq n-1\). By the fiber-dimension
theorem, for a general \(x\in D\),
\[
 \dim C_x
 =\dim C-\dim D
 \geq (2n-1)-(n-1)=n.
\]
Since \(C_x\) is a closed subset of the irreducible \(n\)-dimensional
variety \(Y\), it follows that \(C_x=Y\). The inequality also forces
\(\dim D=n-1\). Since \(C_x=Y\) for general \(x\in D\), taking closures
yields $ C=D\times Y.$
This contradicts the preceding argument.
Hence every irreducible component of
\(\widetilde{\mathcal I}\) dominates \(X\). Interchanging \(X\) and \(Y\) in the preceding argument shows that
\(\widetilde{\mathcal I}\) also dominates \(Y\). Hence it is neither
vertical nor horizontal.

If \(\widetilde{\mathcal I}\) had two distinct irreducible components,
both would dominate \(X\), and their intersections with a general fiber
of \(p_X\) would yield a nontrivial decomposition of that fiber. This
contradicts the irreducibility of the general fiber. Therefore
\(\widetilde{\mathcal I}\) is irreducible.

Let \(C\) denote the unique irreducible component underlying
\(\widetilde{\mathcal I}\). Since \(\widetilde{\mathcal I}\) is an
effective Cartier divisor, we may write
$ \operatorname{div}(\widetilde B)=mC$
for some integer \(m\geq1\). Choose a general \(x\in X\) for which
\(\pi_Y^{-1}(H_{\pi_X(x)})\) is irreducible and generically reduced.
At a generic point \(y\) of this fiber, \(\pi_Y\) is \'etale. Hence
the differential in the \(Y\)-direction of
$ B\bigl(\pi_X(x),\pi_Y(y)\bigr)$
is nonzero. Therefore \(\widetilde{\mathcal I}\) is reduced at
\((x,y)\). Since its support is irreducible, it is reduced at its
generic point, and consequently \(m=1\).
In fact, if \(m>1\), the effective Cartier divisor \(mC\) would be
nonreduced at every point of a dense open subset of \(C\), contradicting
the reduced point constructed above.
\end{proof}

The following lemma is a rational function form of Rosenlicht's
 theorem (cf. Proposition 1.1 in \cite{KKV89}) and we include a short proof for
completeness.

\begin{lemma}
\label{lem:vh-factor}
Let \(X,Y\) be normal irreducible projective varieties.  Call a prime
divisor \emph{vertical} if it is \(D_X\times Y\), and
\emph{horizontal} if it is \(X\times D_Y\), for prime divisors
\(D_X\subset X\), \(D_Y\subset Y\).  If a nonzero meromorphic function
\(R\) on \(X\times Y\) has divisor supported on vertical and horizontal
prime divisors, then
$ R(x,y)=A(x)B(y)$
for meromorphic functions \(A\) on \(X\) and \(B\) on \(Y\).
\end{lemma}

\begin{proof}
Over \(\C\), the product \(X\times Y\) is again normal and
irreducible.  
Choose points \(x_0\in X(\C)\) and \(y_0\in Y(\C)\) outside the
relevant zero and pole divisors, so that \(R(x_0,y_0)\) is finite and
nonzero and the rational functions \(R(\,\cdot\,,y_0)\) and
\(R(x_0,\,\cdot\,)\) are well defined. Define
\[
 S(x,y)
 =
 \frac{R(x,y)R(x_0,y_0)}
      {R(x,y_0)R(x_0,y)}.
\]
The vertical and horizontal orders cancel divisor by divisor, so
\(\operatorname{div}S=0\).  Normality implies that \(S\) and \(S^{-1}\)
are holomorphic.  It follows that \(S\) is constant, and 
\(S(x, y)=S(x_0,y_0)=1\) .  Solving for \(R\), with
$ A(x)=R(x,y_0), 
 B(y)=\frac{R(x_0,y)}{R(x_0,y_0)}$, 
yields the conclusion.
\end{proof}

\begin{proof}[Proof of Theorem~\ref{thm:generalized-balls}]
By Proposition~\ref{lem:effective-algebraicity}, all $f_{i\alpha}$
 in \eqref{eq:minimal-kernel} are algebraic.  Let
\(E_z/\C(z)\) be a common finite extension containing all
\(f_{i\alpha}\), and let \(X\) be the normalization of \(\PP^n\) in
\(E_z\).  Normalization of \(\PP^n\) in a finite extension is finite;
thus \(X\) is normal, irreducible, and projective, and the induced map
\(\pi_X:X\to\PP^n\) is finite and surjective.  Similarly, let
\(E_w/\C(w)\) be the finite extension generated by the coefficientwise
conjugate germs \(f_{i\alpha}^\#\), and let
\(\pi_Y:Y\to\PP^n\) be the corresponding normalization.  
The original algebraic germs select local sheets of \(X\) and \(Y\)
over neighborhoods of the origins.  On these sheets the functions
\(f_{i\alpha}\) and \(f_{i\alpha}^\#\) are the original holomorphic
germs.  The product \(X\times Y\) is normal and irreducible, and each
\(Q_i\) is a meromorphic function on it.  A pole of a summand
\(f_{i\alpha}(x)f_{i\alpha}^\#(y)\) lies over a pole divisor in one of
the two factors; hence all poles of \(Q_i\) are vertical or horizontal.

The local identity \eqref{eq:polarized-dlog} holds on a nonempty open
subset of the product of the selected sheets.  Both sides are
meromorphic one-forms on the irreducible variety \(X\times Y\), so the
identity theorem extends it globally:
\begin{equation}\label{eq:global-dlog}
 \frac{dK}{K}=\sum_i a_i\frac{dQ_i}{Q_i},
\end{equation}
where $ K=\frac{Z_0W_0-\sum_{\nu=1}^nZ_\nu W_\nu}{Z_0W_0}$ in homogeneous coordinates.
Lemma~\ref{lem:prime-incidence} shows that
\(\widetilde{\mathcal I}\) is a prime Weil divisor occurring with
multiplicity one. Since it dominates both factors, it is not a
component of the pole divisor \(Z_0W_0=0\). Therefore
$ \operatorname{ord}_{\widetilde{\mathcal I}}K=1.$
\(\widetilde{\mathcal I}\) is neither vertical nor horizontal, so no \(Q_i\) has a pole
there.  Define
\[
 k_i=\ord_{\widetilde{\mathcal I}}Q_i\in\mathbb Z_{\geq0}.
\]
Taking the residue of \eqref{eq:global-dlog} along
\(\widetilde{\mathcal I}\) yields
$ 1=\sum_i a_i k_i. $ 

Now let \(D\neq\widetilde{\mathcal I}\) be any prime divisor which is
neither vertical nor horizontal.  Then
\[
 \ord_DK=0,\quad \ord_DQ_i\geq0.
\]
The residue of \eqref{eq:global-dlog} at \(D\) is
\[
 0=\sum_i a_i\ord_DQ_i.
\]
Positivity of the \(a_i\)'s forces
\[
 \ord_DQ_i=0\quad\text{for every }i.
\]
Therefore
$ R_i=\frac{Q_i}{K^{k_i}}$ 
has divisor supported only on vertical and horizontal prime divisors.
By Lemma~\ref{lem:vh-factor},
\[
 R_i(x,y)=A_i^\flat(x)B_i^\flat(y).
\]
Let \(x_*\in X\) and \(y_*\in Y\) be the points over \(z=0\) and
\(w=0\) determined by the original selected local branches.  On those
branches, \(f_{i\alpha}(x_*)=f_{i\alpha}^\#(y_*)=0\).  It thus follows from the
normalization \(Q_i(z,0)=Q_i(0,w)=K(z,0)=K(0,w)=1\) that 
$ R_i(x,y_*)=1, R_i(x_*,y)=1.$
These identities hold a priori on nonempty neighborhoods in the
selected sheets; they therefore hold as rational identities on the
irreducible varieties \(X\) and \(Y\).  Moreover, \(R_i\) is
holomorphic and nonzero near \((x_*,y_*)\), so the factors may be
rescaled to be regular and nonzero there.  The rational identity $ R_i(x,y_*)=1$
then forces \(A_i^\flat\) to be constant, and $R_i(x_*,y)=1$
 forces
\(B_i^\flat\) to be constant.
Furthermore, $A_i^\flat (x) B_i^\flat (y)=1$.
  Hence
\[
 Q_i(z,w)=K(z,w)^{k_i}.
\]
Restricting to \(w=\overline z\) and applying
\(-\ii\partial\bar\partial\log\) yields
\[
 G_i^*\eta_{\kappa_i,M_i}=k_i\eta_{0,n}.
\]
This proves
\eqref{eq:generalized-conclusion}. 
By repeating the local argument at
each point of \(U\), we achieve locally constant integer coefficients and these constants must 
 agree because of the connectedness of $U$.
\end{proof}

\begin{remark}
If replacing $\mathbb{B}^n$ in Theorem \ref{thm:generalized-balls} by the generalized ball \(\mathbb B^n_\ell\) with $\ell >0$, the same conclusion still holds by the same argument. 
However  it does not produce a corresponding Theorem \ref{thm:main} with $\mathbb{B}^n$ replaced by \(\mathbb B^n_\ell\) because a positive weighted sum of positive-definite Bergman metrics cannot equal an indefinite metric when \(\ell>0\).
\end{remark}

\section{Proof of the main theorem}

The next lemma is well known to experts and we record it here for completeness.

\begin{lemma}\label{lem:generic-embedding}
Let \(\Omega\subset \C^d\) be an irreducible bounded symmetric
domain in its Harish--Chandra realization, let
\(N_\Omega(z,\overline w)\) be its generic norm, and let
\(\gamma_\Omega\) be its genus.  There exist integers \(p\geq0\),
\(q\geq d\), and holomorphic polynomials
$ P_1,\ldots,P_p,\quad Q_1,\ldots,Q_q$
vanishing at the origin such that
\begin{equation}\label{eq:norm-diagonalization}
 N_\Omega(z,\overline w)
 =1+\sum_{\alpha=1}^pP_\alpha(z)
       \overline{P_\alpha(w)}
   -\sum_{\beta=1}^qQ_\beta(z)
       \overline{Q_\beta(w)}.
\end{equation}
Moreover, the map
$ L_\Omega(z) =(P_1(z),\ldots,P_p(z),Q_1(z),\ldots,Q_q(z))$
is a holomorphic embedding $ L_\Omega:\Omega \rightarrow\B^{p+q}_p$
and satisfies
\begin{equation}\label{eq:generic-isometry}
 L_\Omega^*\eta_{p,p+q}
 =\frac{1}{\gamma_\Omega}\omega_\Omega.
\end{equation}
\end{lemma}

\begin{proof}
The generic norm is a Hermitian polynomial and the Bergman kernel has
the form
\begin{equation}\label{eq:generic-kernel}
 K_\Omega(z,\overline w)
 =C_\Omega N_\Omega(z,\overline w)^{-\gamma_\Omega}.
\end{equation}
These standard generic-norm facts, including the polynomial expansion,
are recorded in \cite[Sections~2.1--2.2]{Roos}.  Moreover,
$N_\Omega(z,0)=N_\Omega(0,\overline w)=1$, and $ N_\Omega(z,\overline z)>0$ for all $z, w \in\Omega.$
Choose a vector \(\mathcal M(z)\) consisting of the finitely many
nonconstant monomials appearing in \(N_\Omega\).  Hermitian symmetry
yields a finite Hermitian matrix \(A\) such that
\[
 N_\Omega(z,\overline w)
 =1+\mathcal M(z)A\mathcal M(w)^*.
\]
The Harish--Chandra realization is circular, and invariance of the
generic norm under its scalar \(S^1\)-action yields
\[
 N_\Omega(e^{\ii t}z,\overline{e^{\ii t}w})
 =N_\Omega(z,\overline w).
\]
Consequently, only terms having equal holomorphic and antiholomorphic
total degree occur, so the coefficient matrix \(A\) is block diagonal
by degree. 
\eqref{eq:norm-diagonalization} is obtained by the standard algebraic operation.

The bidegree-\((1,1)\) block is the negative of a positive definite
Hermitian form on \(\C^d\).  We may therefore choose \(d\) of the negative
coordinates \(Q_\beta\) to be an invertible system of \emph{purely
linear} coordinates on \(\C^d\).  Projection of \(L_\Omega\) onto these
coordinates is a linear isomorphism.  Hence \(L_\Omega\) is globally
injective and has injective differential everywhere, and is therefore
a holomorphic embedding.

Equation \eqref{eq:norm-diagonalization} and positivity of the generic
norm show that
\[
 1-H_p\bigl(L_\Omega(z),
          \overline{L_\Omega(z)}\bigr)
 =N_\Omega(z,\overline z)>0.
\]
Thus \(L_\Omega(\Omega)\subset\B^{p+q}_p\).  Finally,
\eqref{eq:generic-kernel} yields
\[
 \omega_\Omega
 =-\gamma_\Omega\ii\,\partial\bar\partial
       \log N_\Omega(z,\overline z)
 =\gamma_\Omega L_\Omega^*\eta_{p,p+q},
\]
which is \eqref{eq:generic-isometry}.
\end{proof}

\begin{remark}
For the irreducible Type IV domain \(D_m^{IV}\), \(m\geq3\),
\[
 N_{D_m^{IV}}(Z,\overline Z)
 =1-Z\overline Z^{\,t}
   +\frac14|ZZ^t|^2\]
   with $ \gamma_{D_m^{IV}}=m$, 
 and the embedding in Lemma~\ref{lem:generic-embedding}  becomes
\[ L(Z)=\left(\frac12ZZ^t,Z\right):D_m^{IV}\rightarrow\B^{m+1}_1 \]
satisfying 
$ L^*\eta_{1,m+1}=\frac1m\omega_{D_m^{IV}}$. 
This is exactly the construction used in
 Section 3 of \cite{XY20}.
\end{remark}

We are now ready to prove Theorem~\ref{thm:main}.

\begin{proof}[Proof of Theorem~\ref{thm:main}]
For every \(i\), applying Lemma~\ref{lem:generic-embedding} to
\(\Omega_i\), we get
$ L_i:\Omega_i\rightarrow\B^{M_i}_{\kappa_i}$
satisfying $ L_i^*\eta_{\kappa_i,M_i} =\frac1{\gamma_i}\omega_{\Omega_i}$. 
Let $ G_i=L_i\circ F_i$. 
It follows from 
\eqref{eq:main-assumption} that 
\begin{align*}
 \eta_{0,n} =\frac1{n+1}\omega_{\B^n} =\sum_{i=1}^m
   \frac{\lambda_i}{n+1}F_i^*\omega_{\Omega_i} =\sum_{i=1}^m   \frac{\lambda_i\gamma_i}{n+1}   G_i^*\eta_{\kappa_i,M_i}.
\end{align*}
By  Theorem~\ref{thm:generalized-balls},  there exists
integers \(k_i\geq0\) such that
$ G_i^*\eta_{\kappa_i,M_i}=k_i\eta_{0,n}$ and $ \sum_i\frac{\lambda_i\gamma_i}{n+1}k_i=1.$ 
Since \(L_i^*\eta_{\kappa_i,M_i}=\gamma_i^{-1}\omega_{\Omega_i}\),
we obtain
$$ F_i^*\omega_{\Omega_i}
 =\gamma_iG_i^*\eta_{\kappa_i,M_i}
 =\frac{\gamma_i k_i}{n+1}\omega_{\B^n}.$$
If \(k_i=0\), then \(F_i^*\omega_{\Omega_i}=0\).  The Bergman metric of
\(\Omega_i\) is positive definite, so \(dF_i=0\), making \(F_i\)
constant on \(U\).  This contradicts the hypothesis.
Thus \(k_i\geq1\), proving \eqref{eq:main-quantization}.
Now it follows from Mok's extension theorem  
(Theorem 2.2.1 in \cite{M12}) that $F_i$ extends to a proper holomorphic isometric
embedding of \(\B^n\) into \(\Omega_i\).
\end{proof}

\section*{Acknowledgments}
The first author is deeply indebted to Xiaojun Huang for many discussions on the subject over the years. 
The first author was partially supported by Zhejiang Provincial Natural Science
Foundation of China (Grant No.~LQKWL26A0201). The second author is supported in part by the China Postdoctoral Science Foundation (No. 2024M762395).

\noindent Yuan Yuan, yuanyuan@westlake.edu.cn, Institute for Theoretical Sciences, Westlake University,\\
Hangzhou 310024, Zhejiang, China\\
  
\noindent Xu Zhang, xzhangmath@tongji.edu.cn, School of Mathematical Sciences, Tongji University, Shanghai 200092, China.

\end{document}